\documentclass[11pt]{amsart} %

\usepackage{float}

\usepackage{epsfig}
\usepackage{tikz}
\usepackage{array,amsmath, enumerate,  url, psfrag}
\usepackage{amssymb, amsaddr, fullpage}
\usepackage{graphicx,subfigure}
\usepackage{color}

\newtheorem{theorem}{Theorem}[section]
\newtheorem{lemma}[theorem]{Lemma}

\newtheorem{observation}[theorem]{Observation}

\newtheorem{thm}{Theorem}[section]

\newtheorem{definition}[thm]{Definition}

\title{Planar graphs without 7-cycles and butterflies are DP-4-colorable}

\author{Seog-Jin Kim$^{1}$, Runrun Liu$^{2}$,    Gexin Yu$^{2,3}$}

\address{
$^{1}$\small Department of Mathematics Education, Konkuk University, Seoul, 05029,
Korea. \\
$^{2}$\small Department of Mathematics, Central China Normal University, Wuhan, Hubei, China.\\
$^3$\small Department of Mathematics, The College of William and Mary, Williamsburg, VA, 23185, USA.
}

\email{skim12@konkuk.ac.kr (S.-J. Kim), \ 827261672@qq.com (R. Liu), \ gyu@wm.edu (G. Yu)}

\begin{document}
\maketitle

\begin{abstract}
DP-coloring (also known as correspondence coloring) is a generalization of list coloring, introduced by Dvo\v{r}\'ak and Postle in 2017. It is well-known that there are non-4-choosable planar graphs. Much attention has recently been put on sufficient conditions for planar graphs to be DP-$4$-colorable.  In particular, for each $k \in \{3, 4, 5, 6\}$, every planar graph without $k$-cycles is DP-$4$-colorable. In this paper, we prove that every planar graph without $7$-cycles and butterflies is DP-$4$-colorable.  Our proof can be easily modified to prove other sufficient conditions that forbid clusters formed by many triangles.
\end{abstract}

\section{Introduction}

We only consider simple graphs in this article.  A \emph{list assignment} $L$ of a graph $G=(V,E)$  is a mapping that assigns each vertex $u$ of $G$ a set of colors $L(u)$. An \emph{$L$-coloring} of $G$ is a proper coloring $f$ of $V(G)$ such that $f(u) \in L(u)$ for any $u \in V(G)$. A list assignment $L$ is called a \emph{$t$-list assignment} if $|L(u)| \geq t$ for any $u \in V(G)$.
A graph $G$ is \emph{$t$-choosable} if $G$ admits an $L$-coloring for each $t$-list assignment $L$.  The \emph{list-chromatic number} (or the \emph{choice number}) of $G$, denoted by $\chi_{\ell}(G)$, is the minimum integer $t$ such that $G$ is $t$-choosable.

Thomassen \cite{Thomassen} showed that every planar graph is 5-choosable, and Voigt \cite{Voigt} showed that there are planar graphs that are not 4-choosable.  This makes it an interesting question to determine which planar graphs are $4$-choosable.

A graph is said to be \emph{$\ell$-degenerate}  if each of its subgraph contains a vertex of degree at most $\ell$. Let $C_k$ be the cycle with $k$ vertices.   For each $k \in \{3, 5, 6\}$, it is shown that a planar graph without $C_k$ is 3-degenerate, thus $4$-choosable, see \cite{FJMS02, LSX01, Lam, WL02}.  It is further shown in \cite{F09, Lam} that for each $k\in \{4,7\}$, planar graphs without $k$-cycles is $4$-choosable.  The proof for the case of $7$-cycle is quite involved.

Some powerful tools (for example, vertex identification) in coloring are not feasible for list coloring.  In an effort to overcome this,  Dvo\^{r}\'{a}k and Postle \cite{DP18} introduced the notion of $DP$-coloring, which is a generalization of list coloring. By using this notion,  they solved a long-standing conjecture of Borodin~\cite{B13} on list coloring of planar graphs.

\begin{definition}
Let $G$ be a graph with $n$ vertices and let $L$ be a list assignment for $V(G)$. For each edge $uv$ in $G$, let $M_{uv}$ be a matching between the sets $L(u)$ and $L(v)$ and let $\mathcal{M}_L = \{ M_{uv} : uv \in E(G)\}$ (called a matching assignment). Let $G_{L}$ be the graph that satisfies the following conditions
\begin{itemize}
\item each $u\in V(G)$ corresponds to a set $L_u=\{(u,x): x\in L(u)\}$ of vertices in $G_L$;
\item for all $u \in V(G)$, the set $L_u$ forms a clique in $G_L$;
\item if $uv \in E(G)$, then the edges between $L_u$ and $L_v$ are those of $M_{uv}$; and
\item if $uv \notin E(G)$, then there are no edges between $L_u$ and $L_v$.
\end{itemize}
If $G_L$ contains an independent set of size $n$, then $G$ has an {\em $\mathcal{M}_L$-coloring}. The graph $G$ is {\em DP-$k$-colorable} if, for any matching assignment $\mathcal{M}_L$ in which $L(u)\supseteq[k]:=\{1,2,\ldots, k\}$ for each $u \in V(G)$, it has an $\mathcal{M}_L$-coloring. The minimum value of $k$ such that $G$ is DP-$k$-colorable is the {\em DP-chromatic number} of $G$, denoted by $\chi_{DP}(G)$.
\end{definition}

Let $\mathcal{M}_L$ be a matching assignment for $G$. An edge $uv\in E(G)$ is {\em straight} if every $(u,c_1)(v,c_2)\in E(M_{uv})$ satisfies $c_1=c_2$. If all the edges under $M_L$ are straight, then an $M_L$-coloring is exactly a list-coloring.  So DP-coloring is a generalization of list-coloring.  Since more vertices in $L_u$ will only make it easier to find an independent set of size $n$, we may assume that each of $L_u$ has size $k$, namely $L_u=\{(u,i): i\in[k]\}$.  The elements in $L_u$ sometimes are still called colors of $u$.  We may also assume that $M_{uv}$ for each $uv\in E(G)$ is a perfect matching.

Dvo\^{r}\'{a}k and Postle \cite{DP18} observed that Thomassen's proof also implies that planar graphs are DP-$5$-colorable.  As a planar graph without $k$-cycles with $k\in \{3,5,6\}$ is $3$-degenerate, it is also DP-$4$-colorable. Kim and Ozeki \cite{KO18} showed that planar graphs without $4$-cycles are DP-4-colorable.  More sufficient conditions for a planar graph to be $DP$-4-colorable have been found in \cite{CLYZZ18+, KO18, KY17, LL19, LLNSY18+}, and we summarize them below.

\begin{theorem} (\cite{CLYZZ18+, KO18, KY17, LL19, LLNSY18+})
The following planar graphs are DP-4-colorable
\begin{itemize}
\item without $k$-cycles, where $k\in \{3,4,5,6\}$, or
\item without $k$-cycles adjacent to $k'$-cycles, where $(k,k')\in \{(3,4), (3,5), (3,6), (4, 5), (4,6)\}$, and two cycles are adjacent if they share an edge.
\end{itemize}
\end{theorem}

Note that it is conjectured independently in \cite{LSX01, WL02b} that every plane graph without adjacent triangles is 4-choosable and it remains open.  We would like to conjecture that every plane graph without adjacent $4$-cycles is $4$-choosable, or even DP-$4$-colorable.

A {\em cluster} in a plane graph $G$ is a subgraph of $G$ that consists of a minimal set of 3-faces such that no other 3-face is adjacent to 3-faces in the set. It is called a {\em $k$-cluster} if it contains $k$ $3$-faces.   Below is the set of possible clusters with distinct vertices in a plane graph without $7$-cycles (in \cite{F09}, all 23 clusters in such plane graphs are given).

\begin{figure}[H]
\includegraphics[scale=0.7]{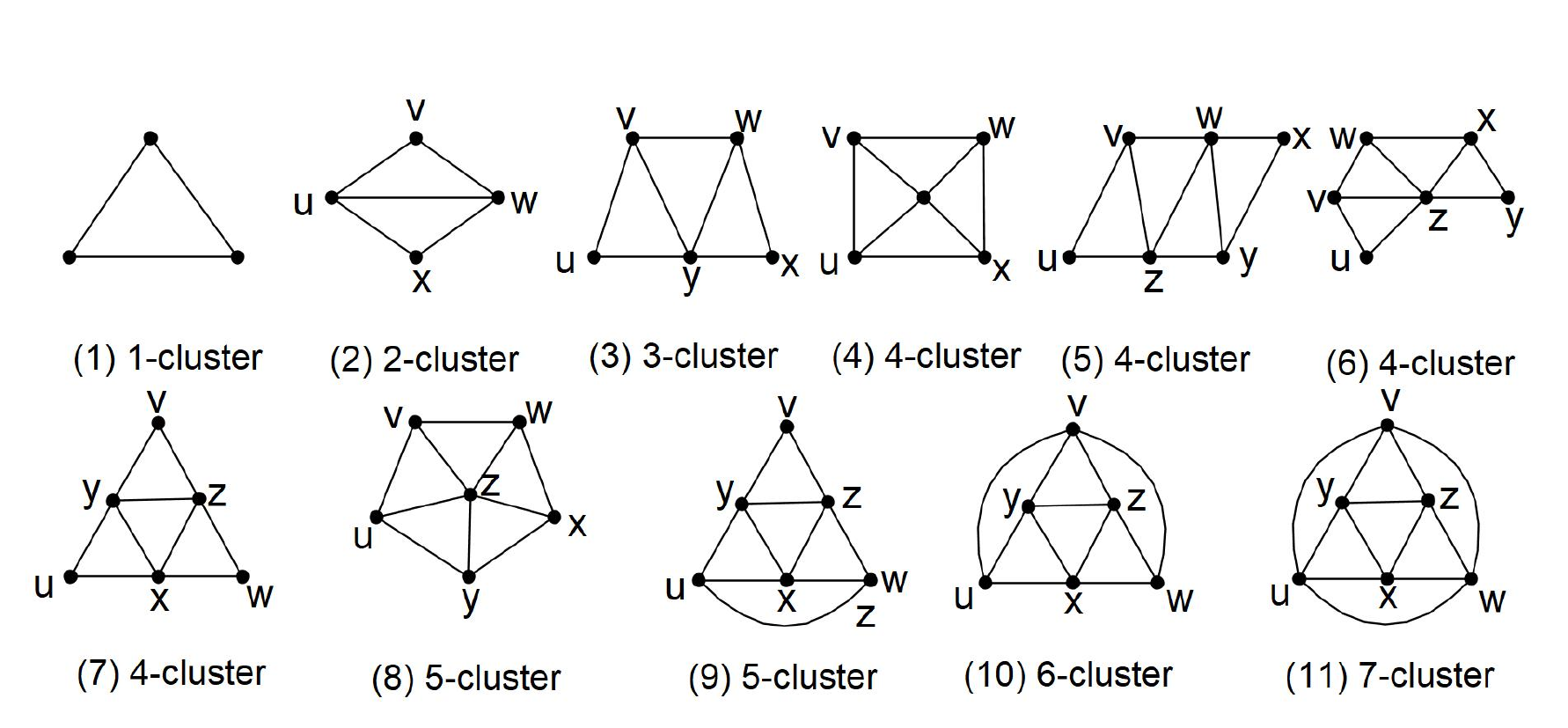}
\caption{All possible clusters with distinct vertices in a plane graph without $7$-cycles.}\label{cluster}
\end{figure}

It turns out that all known results on DP-$4$-coloring of planar graphs use certain kinds of condition to forbid one or more of the given clusters in Figure~\ref{cluster}, especially the clusters with more $3$-faces. What makes the proof of list-4-coloring of planar graphs without $7$-cycles difficult is that one has to take care of all the clusters in the figure.

In this article, we give a sufficient condition for a planar graph to be DP-$4$-colorable. This condition allows the existence of all clusters in Figure~\ref{cluster}.



\begin{theorem}\label{main1}
Every planar graph without 7-cycles and butterflies is DP-4-colorable, where a butterfly is a graph isomorphic to the configuration depicted in Figure~\ref{butterfly}.
\end{theorem}

\begin{figure}[H]
\includegraphics[scale=0.4]{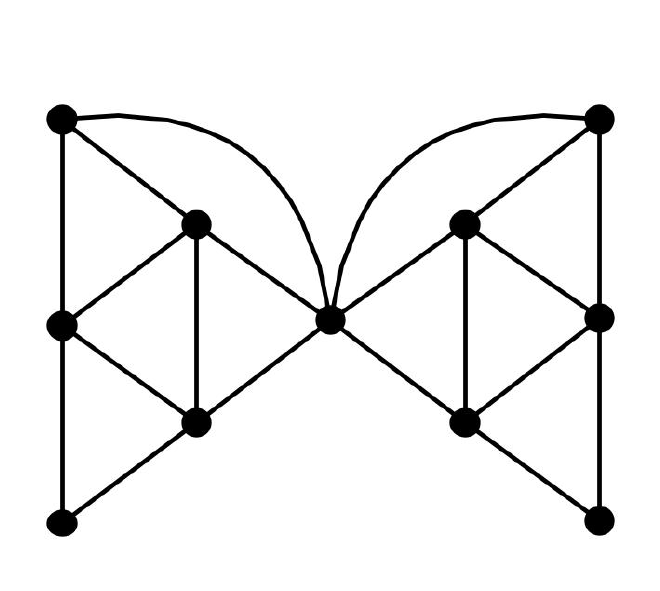}
\caption{A butterfly}
\label{butterfly}
\end{figure}

Without much effort, our proof can be modified to provide more sufficient conditions for a planar graph to be DP-4-colorable. Here is a potential list:  planar graphs without $k$-cycles adjacent to $k'$-cycles, where $(k,k')\in \{(5,5),(5,6),(6,6)\}$.

Let $\mathcal{G}$ be the set of planar graphs without 7-cycles and butterflies.  A $3$-cycle $C$ in a plane graph is {\em bad} if $C$ and interior of $C$ form a $7$-cluster.  So in Figure~\ref{cluster} (11), $uvw$ is a bad $3$-cycle.  A $3$-cycle is called {\em good} if it is not bad.  We actually prove the following stronger result.   


\begin{theorem}\label{main2}
Any DP-4-coloring of a good 3-cycle in plane graph $G\in \mathcal{G}$ can be extended to a DP-4-coloring of $G$.
\end{theorem}

By \cite{KO18}, every planar graph without triangles is DP-4-colorable. So we may assume that $G\in\mathcal{G}$ contains a triangle. In particular, we can always find a good triangle in $G$.  By Theorem~\ref{main2}, we can get a DP-4-coloring of $G$ by extending a DP-4-coloring of a triangle in $G$.

We use a discharging argument in our proof.  This involves the proof of some reducible configurations.  The proof of reducibility of some $6$- and $7$-clusters in Lemma~\ref{555} and \ref{556} involves careful consideration of matching assignments, thus is essentially different from that of the structures in list-coloring.   Also, some $6$-clusters are reducible in list-$4$-coloring but not reducible in DP-$4$-coloring, which makes the addition of forbidding butterflies  necessary for our proof, see the Final Remarks section.  On the other hand, by strengthening Theorem~\ref{main1} to Theorem~\ref{main2}, instead of 23 clusters in \cite{F09}, we only need to discuss 11 clusters (see Figure~\ref{cluster}), so our proof can be modified to give a simplified proof of 4-choosability of planar graphs without $7$-cycles.

The paper is organized as follows. In Section 2, we show the reducible structures useful in our proof. In Section 3, we show the discharging process to complete the proof. In Section 4, we give some examples to show the necessary of adding the butterflies to be forbidden structure.

\section{Reducible configurations}

The following are some notions used in the paper. A $k$-vertex ($k^+$-vertex, $k^-$-vertex, respectively) is a vertex of degree $k$ (at least $k$, at most $k$, respectively). The same notation will be applied to faces and cycles.  An $(\ell_1, \ell_2, \ldots, \ell_k)$-face is a $k$-face $[v_1v_2\ldots v_k]$ with $d(v_i)=\ell_i$ for $1 \leq i \leq k$. Let $C$ be a cycle of a plane graph $G$. We use $int(C)$ (resp. $ext(C)$) to denote the sets of vertices located inside (resp. outside) the cycle $C$. The cycle $C$ is called {\em separating} if both $int(C)$ and $ext(C)$ are nonempty. The next lemma follows immediately from (\cite{DP18}, Lemma 7).

\begin{lemma}\label{straight}
Let $G$ be a graph with a matching assignment $\mathcal{M}_L$. Let $T$ be a subgraph of $G$ which is a tree. Then we may rename $L(u)$ for $u\in T$ to obtain a  matching assignment $\mathcal{M'}_L$ for $G$ such that all edges of $T$ are straight in $\mathcal{M'}_L$.
\end{lemma}

Let $(G,C)$ be a minimal counterexample to Theorem~\ref{main2}, where $C$ is a good $3$-cycle in $G$ that has a DP-4-coloring $\phi_C$. If $C$ is a separating cycle, then any precoloring of $C$ can be extend to $int(C)$ and $ext(C)$, respectively. Then we get a DP-4-coloring of $G$, a contradiction. So we let $G$ be a plane graph so that $C$ is the boundary of the outer face in the rest of this paper. We still denote the outer face by $C$. Call a vertex $v$ in $G$ {\em internal} if $v\notin V(C)$, and a subgraph $H$ in $G$ {\em internal} if $V(H)\cap V(C)=\emptyset$. Assume  that $I'$ is an independent set in $G_L$ with $|I'|<|V(G)|$ that extends $\phi_C$.  For each $v\in V(G)$ with $L_v\cap I'=\emptyset$,   
define $L^*_v=L_v-\{(v,k): (v,k)(u,k)\in E(G_L), u\in N_G(v) \text{ and } (u,k)\in I'\}.$ Intuitively,  $L^*_v$ contains the available colors for $v$ after a partial coloring $I'$.

\begin{lemma}\label{minimum}
Every internal vertex in $G$ has degree at least 4.
\end{lemma}
\begin{proof}
Suppose otherwise that there exists an internal $3^-$-vertex $v$ in $G$. By the minimality of $(G,C)$, $G-v$ has a DP-4-coloring that extends $\phi_C$. Thus there is an independent set $I'$ in $G_L$ with $|I'|=|V(G)|-1$. Since $|L_v|\ge4$ and $v$ is a $3^-$-vertex, we have $|L^*_v|\ge1$. So we can pick a vertex $(v,c)\in L^*_v$ such that $I'\cup\{(v,c)\}$ is an independent set of $G_L$ with $|V(G)|$ vertices, a contradiction.
\end{proof}

\begin{lemma}\label{separating}
G contains no separating good 3-cycles.
\end{lemma}
\begin{proof}
Let $C'$ be a separating good $3$-cycle in $G$. By the minimality of $(G,C)$, $\phi_C$ can be extended to $G-int(C')$. After that, $C'$ is precolored, then again the coloring of $C'$ can be extended to $int(C')$. Thus, we get a DP-4-coloring of $G$, a contradiction.
\end{proof}

\begin{lemma}\label{diamond}
Two internal $(4,4,4)$-faces cannot share exactly one common edge unless they form a $K_4$.
\end{lemma}

\begin{proof}
Suppose for a contradiction that $T_1=uvx$ and $T_2=uvy$ are two internal $(4,4,4)$-faces so that $xy\not\in E(G)$. Let $S=\{u,v,x,y\}$. By the minimality of $(G,C)$, the graph $G-S$ has a DP-4-coloring that extends $\phi_C$. Thus there is an independent set $I'$ in $G_L$ with $|I'|=|V(G)|-4$.  Then $|L^*_u|\ge3,\  |L^*_v|\ge3,\  |L^*_x|\ge2,\  |L^*_y|\ge2.$ So we can select a vertex $(v,c)$ in $L^*_v$ for $v$ such that $L^*_x\setminus \{(x,c): (v,c)(x,c)\in E(G_L)\}$ has at least two available colors. Color $y,u,x$ in order, we can find an independent set $I^*$ with $|I^*|=4$. So $I'\cup I^*$ is an independent set of $G_L$ with $|I'\cup I^*|=|V(G)|$, a contradiction.
\end{proof}

We call a cluster {\em special} if it is one of Figure~\ref{cluster} (7)(9)(10)(11) with three internal $4$-vertices $x,y,z$.  For an internal $5^+$-vertex $v$ in a cluster $H$, we shall call $v$ {\em $i$-type to $H$} if $v$ is incident with  exactly $i$ edges in $H$; furthermore, we call $v$ {\em good} when $H$ is special, and call $v$ {\em bad} otherwise.  For example, vertex $u$ in Figure \ref{cluster} (9) is good when $H$ is special and 3-type to $H$, and vertex $v$ in Figure \ref{cluster} (8) is bad and 3-type to $H$.

\begin{lemma}\label{special5}
Every internal $5$-vertex in $G$ cannot be on two special clusters.
\end{lemma}
\begin{proof}
Let $v$ be an internal $5$-vertex with $N(v)=\{v_i: 1\le i\le5\}$. Suppose otherwise that $v$ is on two special clusters $H_1,H_2$. By symmetry assume that $v_1v_2v_{12}$ in $H_1$ and $v_3v_4v_{34}$ in $H_2$ are internal $(4,4,4)$-faces.  By the minimality of $(G,C)$, the graph $G-\{v_1,v_2,v_3,v_4,v_{12},v_{34},v\}$ has a DP-4-coloring extends $\phi_C$. Thus there is an independent set $I'$ in $G_L$ with $|I'|=|V(G)|-7$. Note that $|L^*_{v_1}|\ge3, |L^*_{v_2}|\ge3, |L^*_{v_3}|\ge3, |L^*_{v_4}|\ge3, |L^*_{v_{12}}|\ge2, |L^*_{v_{34}}|\ge2, |L^*_{v}|\ge 3$. So we can select a vertex $(v_1,c)\in L^*_{v_1}$ such that $(v_1, c)$ has no neighbors in $L^*_{v_{12}}$, and a vertex $(v_3, c')\in L^*_{v_3}$ such that $(v_3, c')$ has no neighbors in $L^*_{v_{34}}$. Color $v, v_4,v_{34}, v_2,v_{12}$ in order, we can find an independent set $I^*$ with $|I^*|=7$. So $I'\cup I^*$ is an independent set of $G_L$ with $|I'\cup I^*|=|V(G)|$, a contradiction.
\end{proof}

We call an internal $6$-vertex $v$ {\em special} if $v$ is good $4$-type to an internal $k_1$-cluster $K_1$ and good $2$-type to a $k_2$-cluster $K_2$ with $6\le k_1\le 7$ and $4\le k_2\le 5$, see Figure~\ref{special-6} for an illustration.

\begin{figure}[H]
\includegraphics[scale=0.35]{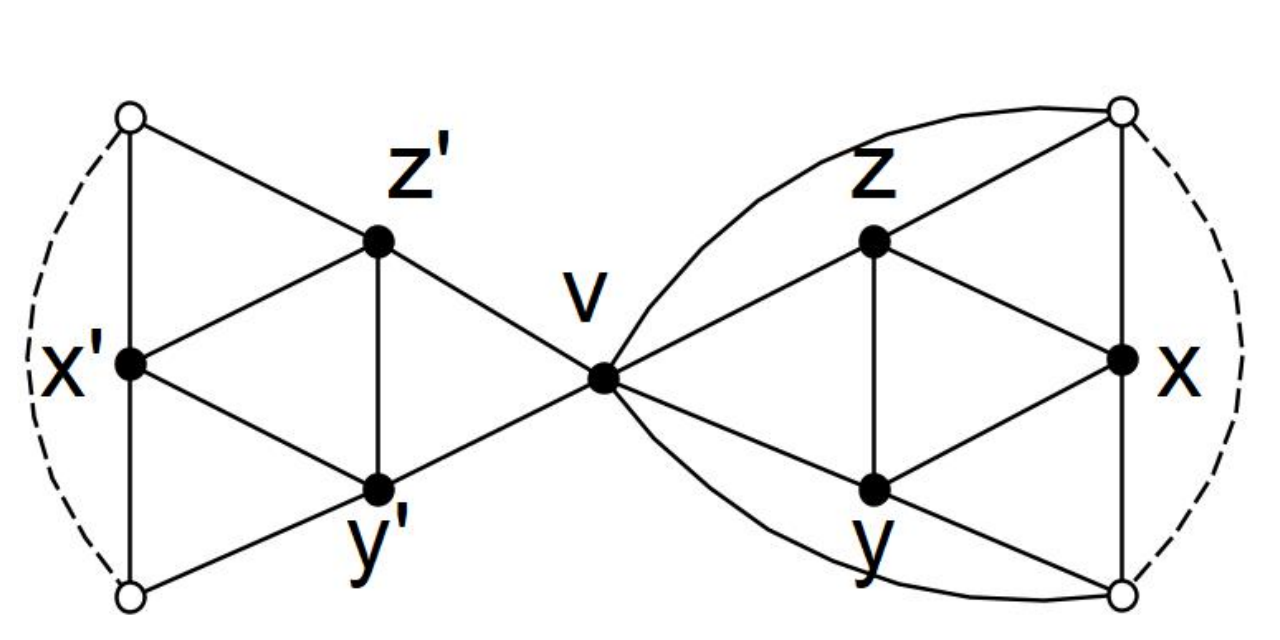}
\caption{A special $6$-vertex $v$ in $G$.}\label{special-6}
\end{figure}

\begin{lemma}\label{special6}
Let $v$ be a special $6$-vertex on a $6^+$-cluster $K_1$ and a $5^-$-cluster $K_2$.  Let $x'y'z'$ be a $(4,4,4)$-face in $K_2$ such that $y',z'\in N(v)$.   Let $\phi_{G'}$ be a DP-coloring of $G'=G-K_1-\{x',y',z'\}$ that extends $\phi_C$. Then among the four colors in $L_v$, one can precolor $v$ with all but at most one color so that $\{x',y',z'\}$ can be colored.
\end{lemma}

\begin{proof}
We may assume that for some color $c$ in $L_v$, $x', y',z'$ cannot be all colored if we color $v$ with $c$. Then  each of $x',y',z'$ has exactly two available colors that induce two triangles in $G_L$.  But then we can precolor $v$ with any other color in $L_v$ such that $x',y',z'$ can be colored in order.
\end{proof}

\begin{lemma}\label{555}
Let $H_6$ be an internal special $6$-cluster isomorphic to Figure~\ref{aaa}(a). If $d(u)=d(w)=5$, then $v$ cannot be a $5^-$-vertex or a special $6$-vertex.
\end{lemma}

\begin{figure}[H]
\includegraphics[scale=0.8]{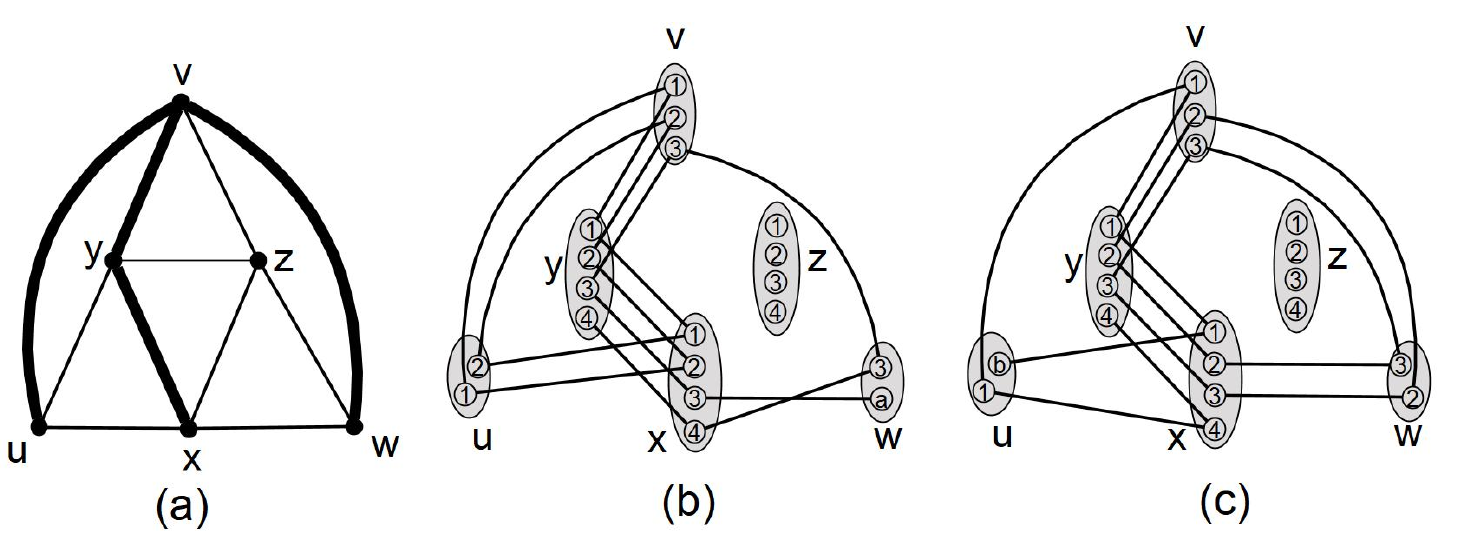}
\caption{}\label{aaa}
\end{figure}

\begin{proof}
By Lemma~\ref{diamond}, $d(v)\ge5$. Suppose that $v$ is a $5$-vertex or a special $6$-vertex. By the minimality of $(G,C)$, $G-H_6$ has a DP-4-coloring that extends $\phi_C$. Thus there is an independent set $I'$ in $G_L$ with $|I'|=|V(G)|-6$.  Note that $|L^*_u|\ge2, |L^*_x|\ge4, |L^*_y|\ge4, |L^*_w|\ge2, |L^*_z|\ge4$.  If $d(v)=5$, then $|L^*_v|\ge 3$.  When $v$ is a special $6$-vertex,  by Lemma~\ref{special6}, we may assume that $v$ has at least three available colors to use when coloring $H_6$.

By Lemma~\ref{straight}, we can assume that the edges $uv, vw, vy, yx$ are straight. If a color $(u,c)$ in $L^*_u$ and a color $(w,c')$ in $L^*_w$ have a common neighbor in $x$, then we select $(u,c), (w,c')$, and $v,y,z,x$ can be colored in order, a contradiction. So we assume that none of the two colors from $L^*_u$ and $L^*_w$ have common neighbors in $L^*_x$. For each $i\in[3]$, if we can choose $(v,i), (x,i)$ such that $(v,i)$ and $(x,i)$ forbidden at most one color at both $u$ and $w$, then $u,w,z,y$ can be colored in order, a contradiction. So we may assume that $(v,i)$ and $(x,i)$ forbidden two colors at $u$ or $w$ for each $i\in[3]$. By symmetry we have two cases, see Figure~\ref{aaa}(b)(c). Choose $c=2$ at Case (b) and choose $c=b$ at Case (c). If $(u,c)(y,1)\in E(G_L)$, then we select $(u,c), (v,1)$, then $w,x,z,y$ can be colored in order.  If $(u,c)(y,j)\in E(G_L)$ for $j\in\{2,3,4\}$, then we select $(u,c), (x,j)$, then $w,v,z,y$ can be colored in order, a contradiction.
\end{proof}

\begin{figure}[H]
\includegraphics[scale=0.8]{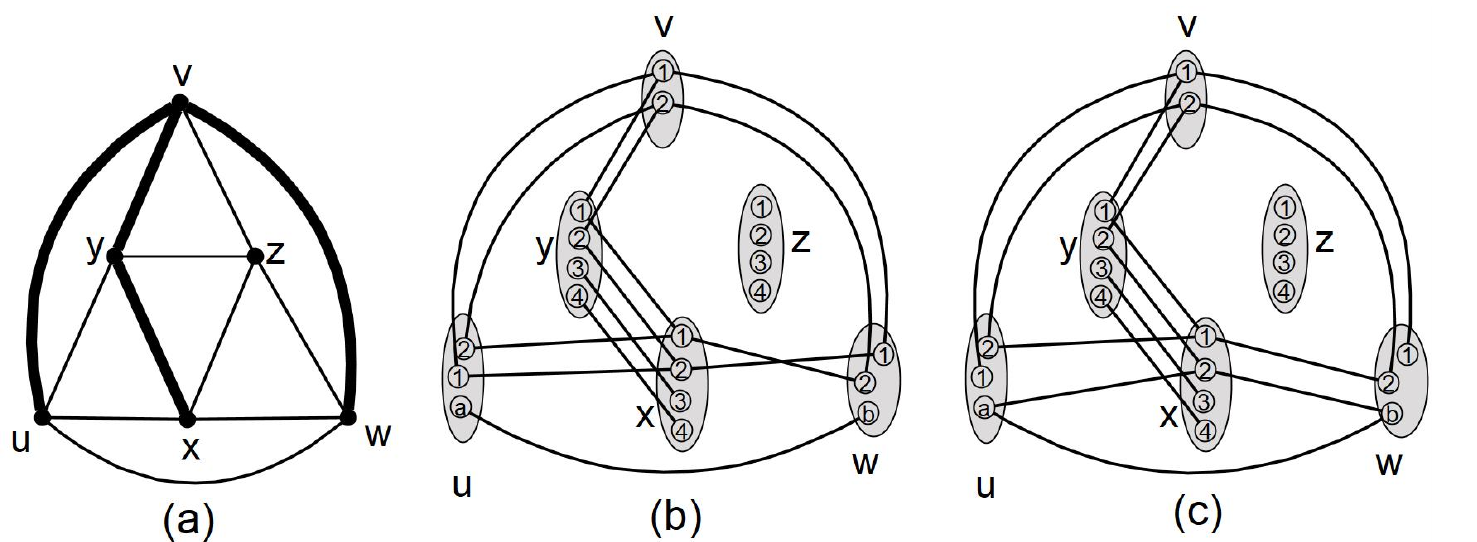}
\caption{} \label{bbb}
\end{figure}

\begin{lemma}\label{556}
Let $H_7$ be an internal $7$-cluster, see Figure~\ref{bbb}(a). If $\max\{d(u),d(v),d(w)\}\le6$, then $H_7$ is incident to at most one  $5$-vertex or special $6$-vertex.
\end{lemma}

\begin{proof}
Suppose otherwise, by symmetry let $u,w$ be $5$-vertices or special $6$-vertices. By the minimality of $(G,C)$, the graph $G-H_7$ has a DP-4-coloring that extends $\phi_C$. Thus there is an independent set $I'$ in $G_L$ with $|I'|=|V(G)|-6$. Note that $ |L^*_v|\ge2$, and $|L^*_x|, |L^*_y|, |L^*_z|\ge4$. By Lemma~\ref{special6}, each of $u,w$ has at least three available colors to use when coloring $H_7$.

By Lemma~\ref{straight}, we can assume that the edges $uv, vw, vy, yx$ are straight.  Let $(v,1), (v,2)\in L^*(v)$ and by straightness, for $i\in \{1,2\}$,  let $(v,i)(u,i), (v,i)(y,i), (y,i)(x,i), (v,i)(w,i)\in E(G_L)$.  We further assume that $(y,3)(x,3), (y,4)(x,4)\in E(G_L)$.   Let $(u,a)\in L^*(u)-\{(u,1), (u,2)\}$ and $(w,b)\in L^*(w)-\{(w,1), (w,2)\}$.

We may assume that $(u,a)(w,b)\in E(G_L)$.  For otherwise, we extend $I'$ to include $(u,a),(w,b)$. Then $|L^*_x|, |L^*_v|\ge2, |L^*_y|, |L^*_z|\ge3$. We pick a color in $L^*_y$ such that $v$ still has two available colors. Then $x,z,v$ can be colored in order, a contradiction.

For each $i\in[2]$, if $(u,i)(x,i)\in E(G_L)$, then we select $(x,i),(v,i)$, then $w,u,z,y$ can be colored in order. So $(u,i)(x,i)\notin E(G_L)$. Similarly, $(w,i)(x,i)\notin E(G_L)$.

We observe that $(x,1), (x,2)$ must have a neighbor in $L^*_u$.  Suppose by symmetry that $(x,1)$  has no neighbors in $L^*_u$.  Then we select $(x,1),(v,1)$ and $w,z,u,y$ can be properly colored in order, a contradiction.  Thus either $(x,1)(u,2)\in E(G_L)$ or $(x,1)(u,a)\in E(G_L)$.  Note that if $(x,1)(u,a)\in E(G_L)$, then we have $(x,2)(u,1)\in E(G_L)$.  Therefore, we have either $(x,1)(u,2)\in E(G_L)$ or $(x,2)(u,1)\in E(G_L)$.  By symmetry of colors, we may assume that $(x,1)(u,2)\in E(G_L)$.    Then either $(x,2)(u,1)\in E(G_L)$ or $(x,2)(u,a)\in E(G_L)$.

Consider the case $(x,2)(u,1)\in E(G_L)$.  If $(x,1)(w,b)\in E(G_L)$, then we select $(u,2),(w,b)$. So $|L^*_v|\ge1, |L^*_x|,|L^*_y|, |L^*_z|\ge3$. Then $v,y,z,x$ can be colored in order, a contradiction. So $(x,1)(w,b)\notin E(G_L)$.  Similarly, we have $(x,2)(w,b)\notin E(G_L)$. Therefore we have $(x,1)(w,2),(x,2)(w,1)\in E(G_L)$ since we showed that $(w,i)(x,i)\notin E(G_L)$ for $i \in \{1, 2\}$, which gives us Figure~\ref{bbb}(b).

Consider the case of $(x,2)(u,a)\in E(G_L)$.  If $(x,1)(w,b)\in E(G_L)$, we select
$(u, 2), (w, b)$, then $v,y,z,x$ can be colored in order, a contradiction.  Thus $(x,1)(w,2)\in E(G_L)$. Next, if $(x,2)(w,1)\in E(G_L)$, we select $(u, a), (w, 1)$, then $v,y,z,x$ can be colored in order, a contradiction.   Thus $(x,2)(w, b)\in E(G_L)$.  So we have $(x,1)(w,2)\in E(G_L)$, which gives us Figure~\ref{bbb}(c).


Now, we show how to color Figure~\ref{bbb} (b) and (c).

{\bf Case 1:} $(u,2)(y,1)\in E(G_L)$.   In this case,  we select $(u,2), (v,1), (w,b)$ in both (b) and (c).  Then $|L^*_x|,|L^*_z| \ge 2, |L^*_y|\ge3$, and $x,z,y$ can be colored in order, a contradiction.

{\bf Case 2:} $(u,2)(y,j)\in E(G_L)$ for some $j \in \{3, 4\}$. In this case,  we select $(x,1), (v,2), (y, j)$ in both (b) and (c).  Then  $|L^*_u|, |L^*_w|\ge2, |L^*_z|\ge1$, and $z,w,u$ can be colored in order, a contradiction.

{\bf Case 3:} $(u,2)(y,2)\in E(G_L)$.  In this case, for (b),  we select $(v,1), (x,2), (u, 2)$ and $w,z,y$ can be colored in order.  For (c), either $(u,a)(y,j)\in E(G_L)$ for some $j \in \{1, 2\}$,  or  $(u,a)(y,j)\in E(G_L)$ for some $j \in \{3, 4\}$. In the former case, we select $(u,a), (v,2), (x, 1)$ and color $w,z,y$ in order;  in the latter case, we select $(u,a),(x,j)$ and color $w,v, z,y$ in order.

This completes the proof of Lemma \ref{556}.
\end{proof}

\section{Discharging procedure}

{\bf Proof of Theorem~\ref{main2}}. We are now ready to  complete the proof of Theorem~\ref{main2} by a discharging procedure.  Let $x\in V(G)\cup F(G)-\{C\}$ has an initial charge of $\mu(x)=d(x)-4$, and $\mu(C)=d(C)+4$. By Euler's Formula, $\sum_{x\in V\cup F}\mu(x)=0$.
Let $\mu^*(x)$ be the charge of $x\in V\cup F$ after the discharge procedure. To obtain a contradiction, we shall prove that $\mu^*(x)\ge 0$ for all $x\in V(G)\cup F(G)$ and $\mu^*(C)>0$.

\begin{observation}\label{obs}
Every good $5$-vertex is on at most one special cluster.
\end{observation}

\begin{proof}
If a good $5$-vertex $v$ is on two special clusters, then $v$ has four neighbors $v_1, v_2, v_3, v_4$ such that $v_1v_2$ and $v_3v_4$ are both on $(4,4,4)$-faces.  This contradicts Lemma~\ref{special5}.
\end{proof}

\medskip
\noindent {\bf The discharging rules:}

\begin{enumerate}[(R1)]
\item Each $5^+$-face $f$ gives $\frac{1}{2}$ to each adjacent $3$-face, and moreover, if $e=uv\in E(f)$ is not on a $3$-face, then $f$ gives $\frac{1}{4}$ to each of $u$ and $v$. Each internal $4$-vertex having at leasts three incident edges in a cluster gives the charge it obtained from incident $5^+$-faces to the cluster.

 \item Let $H_k$ be a $k$-cluster with $k\in [5]$. Then $H_k$ gets $\frac{1}{2}$ from each incident $2$-type $5^+$-vertex that is either good or on a $4$-face adjacent to $H_k$, $\frac{1}{2}$ from each $3$-type $5$-vertex that is bad to $H_k$ or incident $3$-type $6^+$-vertex,  $1$ from each good $3$-type $5$-vertex, and $\frac{3}{2}$ from each incident $4$-type $5^+$-vertex.  Note that $3^+$-clusters cannot be adjacent to $4$-faces.

\item A $6$-cluster $H_6$ gets $\frac{1}{2}$ from each incident bad $3$-type $5$-vertex, $1$ from each incident good  $3$-type $5$-vertex, $\frac{3}{2}$ from each incident $4$-type $5^+$-vertex or $3$-type $6^+$-vertex, $2$ from each incident good $4$-type $6^+$-vertex if $H_6$ contains two $3$-type $5$-vertices.

\item  A $7$-cluster $H_7$ gets $\frac{3}{2}$ from each incident $5$-vertex or special $6$-vertex, $2$ from each incident non-special $6$-vertex, and $\frac{5}{2}$ from each incident $7^+$-vertex.

\item The outer-face $C$ gets $\mu(v)$ from each incident vertex and gives $1$ to each non-internal $3$-face.
\end{enumerate}

{\bf{Remark:}} (a) Let $v$ be a $4^+$-vertex that incident with four consecutive faces $f_1,f_2,f_3,f_4$ in order. If $d(f_1),d(f_4)\ge8$ and $d(f_2)=d(f_3)=3$, then by (R1)(R2) $v$ gives $\frac{1}{2}$ to the cluster $H_k$ containing $f_2,f_3$, when $k\in[5]$.

(b) Let $H$ be a cluster and $v\in H$ be an internal $5^+$-vertex. By (R2)-(R4) $v$ gives at most $\frac{1}{2}$ to $H$ if $v$ is $2$-type to $H$ (note that $6^+$-clusters contain no 2-type vertices), at most 1 to $H$ if $v$ is $3$-type $5$-vertex to $H$, $\frac{3}{2}$ to $H$ if $v$ is a $4$-type $5$-vertex or $3$-type $6^+$-vertex of $H$, at most $2$ to $H$ if $v$ is a $4$-type $6$-vertex of $H$, at most $\frac{5}{2}$ to $H$ if $v$ is a $4$-type $7^+$-vertex of $H$.

\medskip

First we check the final charge of vertices in $G$. Let $v$ be a vertex in $G$. If $v\in V(C)$, then $\mu^*(v)\ge0$ by (R5). So we may assume that $v\notin V(C)$.  By Lemma~\ref{minimum} $d(v)\ge4$. If $d(v)=4$, then by (R1), $\mu^*(v)=\mu(v)=d(v)-4=0$.

Suppose $d(v)=5$.  Then $v$ is on at most two clusters.  If $v$ is on at most one cluster, then by (R1)-(R3), $v$ gives out more than $1$ to the cluster only if $v$ is a $4$-type vertex to the cluster.  But in this case,  each face adjacent to the cluster is a $8^+$-face since $G$ contains no $7$-cycles. So $v$ is incident with two consecutive $8^+$-faces, thus obtains $\frac{1}{4}\cdot 2=\frac{1}{2}$ by (R1). So $\mu^*(v)\ge 5-4-\frac{3}{2}+\frac{1}{2}=0$.   Now let $v$ be on two clusters.  Then $v$ is $3^-$-type to one cluster and $2$-type to the other cluster, and none of the cluster can be a $7$-cluster.  By (R2)(R3), $v$ gives more than $\frac{1}{2}$ to one cluster when $v$ is a good $3$-type vertex.  In this case, $v$ is a $2$-type $5$-vertex to the other cluster, say $H$. It follows that $v$ cannot be on a $4$-face, and by Observation~\ref{obs}, $v$ is not good to $H$.  So by (R2)(R3), $v$ gives no charge to $H$.  Therefore, $\mu^*(v)\ge 5-4-1=0$.


Now let $d(v)\ge 6$.  Assume that $v$ gives $\frac{1}{2}$ to $x_2$ clusters, $\frac{3}{2}$ to $x_3$ clusters, $2$ to $x_4$ clusters, and $\frac{5}{2}$ to $x_4'$ clusters.  Since $G$ contains no butterfly, $x_4+x_4'\le 1$.  Then by (R2)-(R4),   $d(v)\ge 2x_2+3x_3+4(x_4+x_4')$ and $\mu^*(v)\ge d(v)-4-(\frac{1}{2}x_2+\frac{3}{2}x_3+2x_4+\frac{5}{2}x_4')$.  It follows that
\begin{align*}
\mu^*(v)&\ge d(v)-4-\frac{1}{2}(2x_2+3x_3+4x_4+4x_4'-x_2+x_4')\\
&\ge d(v)-4-\frac{1}{2}d(v)+\frac{1}{2}(x_2-x_4')=\frac{1}{2}(d(v)-8+x_2-x_4').
\end{align*}

So if $d(v)\ge 8$, then $\mu^*(v)<0$ only if $d(v)=8$, $x_2=0$ and $x_4'=1$. In this case, $v$ is on a $7$-cluster and $3^+$-type to a $5^-$-cluster, so by (R2)(R4), $\mu^*(v)\ge 8-4-\frac{5}{2}-\frac{3}{2}=0$.

For $d(v)=7$,  $\mu^*(v)\ge \frac{1}{2}(x_2-x_4'-1)$.  If $x_2=0$, then $v$ is $3^+$-type to its clusters so $v$ is on at most two clusters. Note that $v$ is on at most one $6^+$-cluster since $G$ has no butterfly. If $v$ is on a $6^+$-cluster, then by (R2)(R3), either $v$ is $4$-type to a $5^-$-cluster, thus $\mu^*(v)\ge 7-4-\frac{3}{2}-\frac{3}{2}=0$, or $v$ is $3^-$-type to a $5^-$-cluster, which gets at most $\frac{1}{2}$ from $v$ since $G$ contains no butterflies. So $\mu^*(v)\ge 7-4-(\frac{1}{2}+\frac{5}{2})=0$. If $v$ is not on any $6^+$-cluster, then by (R2), $v$ gives at most $\frac{3}{2}$ to each of the two clusters. So $\mu^*(v)\ge 7-4-\frac{3}{2}-\frac{3}{2}=0$.   Now let $x_2\ge 1$, then $\mu^*(v)<0$ only if $x_2=x_4'=1$ since $x_4'\le1$.  It implies $x_3=0$.  By (R2)(R4), $\mu^*(v)\ge 7-4-\frac{5}{2}-\frac{1}{2}=0$.

Let $d(v)=6$.  By (R2)-(R4), $x_4'=0$. So $\mu^*(v)\ge \frac{1}{2}(x_2-2)$.  We may assume that $x_2\le 1$.  Assume first that $x_2=0$. Then $v$ is $3^+$-type to at most two clusters, and if $v$ is $4$-type to a cluster, then $v$ is only on one cluster.  Thus by (R2)(R3), $\mu^*(v)\ge 6-4-\max\{\frac{1}{2}+\frac{3}{2}, 2\}=0$. So let $x_2=1$.  Then $v$ is on at most one $3^+$-cluster, and in particular,  $x_3+x_4\le 1$.  Thus $\mu^*(v)<0$ only if $x_4=1$ and $x_3=0$.  So $v$ is not on a $4$-face, and by (R2)-(R4), $v$ is good $2$-type to one $5^-$-cluster and good $4$-type to one $6$-cluster which contains two $3$-type $5$-vertices, contrary to Lemma~\ref{555}.

\medskip

Now we check the final charge of faces in $F(G)-\{C\}$. Let $f$ be a face in $G$. Since $G$ contains no $7$-cycles,  a $5$-face is adjacent to at most one $3$-face and a $6$-face is not adjacent to $3$-faces.  By (R1),  $\mu^*(f)\ge0$ when $f$ is a $4^+$-face. So we only need to consider the final charge of $3$-faces. Let $d(f)=3$.   Since $G$ contains no $7$-cycles or separating $3$-cycles,  $f$ must be in a $k$-cluster $H_k$ for some $k\in[7]$, which are depicted in Figure~\ref{cluster}.  Let $\mu(H_k)=\mu(f_1)+\mu(f_2)+\cdots+\mu(f_k)=-k$,
where $f_1, f_2, \cdots,  f_k$ are $3$-faces in $H_k$,  and define $\mu^*(H_k)=\mu^*(f_1)+\mu^*(f_2)+\cdots+\mu^*(f_k)$.  We shall show that $\mu^*(H_k)\geq 0$ for each $H_k$, which would imply that all $3$-faces can get enough charge to have nonnegative final charges.

{\bf Case 1. $1$-cluster $H_1$.} Since $G$ contains no $7$-cycles,  $H_1$ shares at least two edges with $5^+$-faces. So $H_1$ either gets $1$ from $C$ by (R5) or $\frac{1}{2}\cdot2$ from adjacent $5^+$-faces by (R1). Thus, $\mu^*(H_1)\ge-1+\frac{1}{2}\cdot2=0$.

\smallskip

{\bf Case 2. $H_2$ is a $2$-cluster isomorphic to Figure~\ref{cluster} (2).} Since $G$ contains no $7$-cycles, $H_2$ shares edges only with $4$-faces or $8^+$-faces; furthermore, $H_2$ shares at least two edges with $8^+$-faces, which gives at least $1$ to $H_2$ by (R1). So we only need 1 more charge to make $\mu^*(H_2)\ge 0$. We may assume that $V(H_2)\cap V(C)=\emptyset$, for otherwise, by (R5) $H_2$ gets at least $1$ from $C$. If $H_2$ shares at least three edges with $8^+$-faces, then by symmetry say $uv,ux$ are on $8^+$-faces. So $H_2$ gets $\frac{1}{2}$ from $u$ by Remark (a) and extra $\frac{1}{2}$ from adjacent $8^+$-faces by (R1). If $H_2$ shares exactly two edges with $8^+$-faces, then by symmetry $wv,wx$ are on $4$-faces or $wv, ux$ are on $4$-faces.
In the former case, $d(w)\ge 5$ and thus $w$ gives $\frac{1}{2}$ to $H_2$ by (R2), and by Remark (a), $u$ gives at least $\frac{1}{2}$ to $H_2$ as well.   Consider the latter case. If both $u,w$ are $4$-vertices, then each of them transfers $\frac{1}{4}$ they got from $8^+$-faces to $H_2$, and by Lemma~\ref{diamond}, $v$ or $x$ is a $5^+$-vertex, which by (R2), gives at least $\frac{1}{2}$ to $H_2$, therefore $\mu^*(H_2)\ge -2++\frac{1}{2}\cdot 2+\frac{1}{4}\cdot 2+\frac{1}{2}=0$. By symmetry, assume that $d(u)\ge 5$.  Now by (R2), $u$ gives $\frac{1}{2}$ to $H_2$, and each of $v,w,x$ gives $\frac{1}{2}$ (if it is a $5^+$-vertex) or at least $\frac{1}{4}$ (if it is a $4$-vertex) to $H_2$.  Therefore, $mu^*(H_2)\ge -2+\frac{1}{2}\cdot 2+\frac{1}{2}+\frac{1}{4}\cdot 3>0$.
\smallskip

Let $H_k$ be a $k$-cluster with $k\ge3$. All faces adjacent to $H_k$ are $8^+$-faces since $G$ has no $7$-cycles.

{\bf Case 3. $H_3$ is a 3-cluster isomorphic to Figure~\ref{cluster} (3).} By (R1), $H_3$ gets $\frac{1}{2}\cdot5$ from adjacent $8^+$-faces. In addition, $H_3$ either gets at least $1$ from $C$ by (R5) or $\frac{1}{2}$ from each of $v$ and $w$ by (R2) (see Remark (a)). So $\mu^*(H_3)\ge-3+\frac{1}{2}\cdot5+1>0$.
\smallskip

{\bf Case 4. $H_4$ is a $4$-cluster isomorphic to one in Figure~\ref{cluster} (4)-(7).}

Case 4.1. $H_4$ is isomorphic to one of Figure~\ref{cluster} (4)-(6). By (R1),  $H_4$ gets $\frac{1}{2}\cdot4,\frac{1}{2}\cdot6,\frac{1}{2}\cdot6$ from adjacent $8^+$-faces in (4),(5),(6), respectively. If any vertex in $H_4$ is on $C$, then by (R5), $H_4$ gets at least $2,1,1$ from $C$ in (4),(5),(6) respectively, so $\mu^*(H_4)\ge 0$. If $H_4$ is internal, then by (R2), in (4) it gets $\frac{1}{2}$ from each of $u,v,w,x$, in (5) it gets $\frac{1}{2}$ from each of $v$ and $y$, and in (6), it gets $\frac{1}{2}$ from each of $v,w,x$. In any case, $\mu^*(H_4)\ge-4+\min\{\frac{1}{2}\cdot4+2,\frac{1}{2}\cdot6+\frac{1}{2}\cdot 2,\frac{1}{2}\cdot 6+\frac{1}{2}\cdot 3\}=0$.

Case 4.2. $H_4$ is isomorphic to Figure~\ref{cluster}  (7). By (R1), $H_4$ gets $\frac{1}{2}\cdot6$ from adjacent $8^+$-faces.  We need to find $1$ more charge to make $\mu^*(H_4)\ge 0$. We may assume that $H_4$ is internal, for otherwise, $H_4$ gets at least $1$ from $C$ by (R5). If one of $x,y,z$ is a $5^+$-vertex, say $x$, then $H_4$ gets $\frac{3}{2}$ from $x$ by (R2). Otherwise, by Lemma~\ref{diamond}, $d(u),d(v),d(w)\ge5$. By (R2), each of $u,v,w$ gives $\frac{1}{2}$ to $H_4$.

\smallskip

{\bf Case 5. $H_5$ is a $5$-cluster isomorphic to Figure~\ref{cluster} (8) or (9).}

Case 5.1. $H_5$ is isomorphic to Figure~\ref{cluster}  (8). By (R1)(R2), $H_5$ gets $1$ from $z$ and $\frac{1}{2}\cdot5$ from adjacent $8^+$-faces. If $H_5$ is not internal, then it gets at least $2$ from $C$ by (R5), thus $\mu^*(H_5)\ge -5+1+\frac{5}{2}+2>0$. If $H_5$ is internal, then $H_5$ gets at least $\frac{1}{2}$ from each of $u, v, w, x, y$ by (R2), thus $\mu^*(H_5)\ge -5+1+\frac{5}{2}+\frac{1}{2}\cdot 5>0$.


Case 5.2. $H_5$ is isomorphic to Figure~\ref{cluster} (9). By (R1), $H_5$ gets $\frac{1}{2}\cdot5$ from adjacent $8^+$-faces.  We may assume that both $y$ and $z$ are internal, for otherwise, $H_5$ gets $3$ from $C$ by (R5) and $\mu^*(H_5)\ge -5+\frac{5}{2}+3>0$. First assume that both $y$ and $z$ are internal $4$-vertices. Then by Lemma~\ref{diamond}, $u,v,w$ are $5^+$-vertices or on $C$. Note that if $u,w$ are internal $5$-vertices, then they are good $3$-type $5$-vertices to $H_5$. By (R2)(R5), $H_5$ gets at least $2$ through $u,w$ and at least $\frac{1}{2}$ from $v$ by (R2), and $\mu^*(H_5)\ge -5+\frac{5}{2}+2+\frac{1}{2}=0$.  Now by symmetry, let $d(y)\ge 5$. By (R2), $H_5$ gets $\frac{3}{2}$ from $y$. Now, whether $u$ and $w$ are $4$-vertices, $5^+$-vertices, or on $C$, $H_5$ gets at least $\frac{1}{2}$ from each of them, so $\mu^*(H_5)\ge -5+\frac{5}{2}+\frac{3}{2}+\frac{1}{2}\cdot 2=0$.


\smallskip

{\bf Case 6. $H_6$ is a $6$-cluster isomorphic to Figure~\ref{cluster}  (10).} By (R1), $H_6$ gets $\frac{1}{2}\cdot4$ from adjacent $8^+$-faces. So we need to find another $4$ to make $\mu^*(H_6)\ge0$. Since $C$ is a good $3$-cycle, $|V(H_6)\cap V(C)|\le2$. If $|V(H_6)\cap V(C)|=2$, then $H_6$ gets at least $4$ from $C$ by (R5). If $|V(H_6)\cap V(C)|=1$, then by symmetry either $u\in V(C)$ or $v\in V(C)$. In the case of $u\in V(C)$, $H_6$ gets $2$ from $C$ by (R5). By Lemma~\ref{diamond}, at least one of $v,x$ is a $5^+$-vertex, by symmetry say $v$. Then by (R3), $H_6$ gets $\frac{3}{2}$ from $v$. If $d(x)=4$, then by Lemma~\ref{diamond} $d(w)\ge5$, so $w$ gives at least $1$ to $H_6$ by (R3);  If $d(x)\ge5$, then $x$ gives $\frac{3}{2}$ to $H_6$ by (R3);  So $H_4$ gets at least $2+\frac{3}{2}+1>4$. Now consider the case $v\in V(C)$. By (R5), $H_6$ gets $3$ from $C$. By Lemma~\ref{diamond} at least one of $u,x,w$ is a $5^+$-vertex, which gives at least $1$ to $H_6$ by (R3). So we assume that $|V(H_6)\cap V(C)|=0$. By Lemma~\ref{diamond}, at least one of $x$ and $v$ is a $5^+$-vertex. If both $x$ and $v$ are $5^+$-vertices, then by (R3), $H_6$ gets $\frac{3}{2}$ from each of $x$ and $v$, and $\frac{1}{2}$ from each of $u$ and $w$. So $H_6$ gets  $\frac{3}{2}\cdot2+\frac{1}{2}\cdot2=4$. If one of $x$ and $v$, by symmetry say $x$ is $4$-vertex, then by Lemma~\ref{diamond}, $d(u), d(v), d(w)\ge 5$ and by Lemma~\ref{555}, one of $u,v,w$ has degree more than $5$. If $d(v)\ge 6$, then by (R3) $H_6$ gets $2$ from $v$ and $1$ from each of $u$ and $w$; and $d(u)\ge 6$, then by (R3), $H_6$ gets at least $\frac{3}{2}$ from each of $v$ and $u$ and at least $1$ from $w$; and $d(w)\ge 6$, then by (R3), $H_6$ gets at least $\frac{3}{2}$ from each of $v$ and $w$ and at least $1$ from $u$. In any case, $H_6$ gets at least $\min\{1+1+2,1+\frac{3}{2}+\frac{3}{2}\}=4$.


\smallskip

{\bf Case 7. $H_7$ is a $7$-cluster isomorphic to Figure~\ref{cluster} (11).} By (R1), $H_7$ gets $\frac{1}{2}\cdot3$ from adjacent $8^+$-faces.  So we need to find another $\frac{11}{2}$ to make $\mu^*(H_7)\ge0$. By Lemma~\ref{diamond} each of $u,v,w$ is either on $C$ or an internal $5^+$-vertex; by (R3), $H_7$ gets at least $\frac{3}{2}$ from each of $u,v,w$ when it is internal.   Since $C$ is a good $3$-cycle, $|V(H_7)\cap V(C)|\le2$. If $|V(H_7)\cap V(C)|\ge1$, then by (R5), $H_7$ gets at least $3$ from $C$.   So $H_7$ gets at least $\min\{3+\frac{3}{2}\cdot 2, 5+\frac{3}{2}\}>\frac{11}{2}$. So we assume that $H_7$ is internal.   If one of $u,v,w$ is a $7^+$-vertex, then by (R4), $\mu^*(H_7)\ge -7+\frac{1}{2}\cdot 3+\frac{5}{2}+\frac{3}{2}\cdot 2=0$. So we assume that $\max\{d(u),d(v), d(w)\}\le 6$. By Lemma~\ref{556}, $H_7$ contains at most one $5$-vertex or special $6$-vertex, so by (R4), $\mu^*(H_7)\ge -7+\frac{1}{2}\cdot 3+\frac{3}{2}+2\cdot 2=0$.


\medskip

To finish the proof, we show that the outer face $C$ has positive final charge.  Let $f_3$ be the number of non-internal $3$-faces. Let $E(C, V(G)-C)$ be the set of edges between $C$ and $V(G)-C$ and let $e$ be its size. Note that $e \geq f_3$. Then by (R5),
\begin{align*}
\mu^*(C)&=3+4+\sum_{v\in V(C)} (d(v)-4)-f_3=7+\sum_{v\in V(C)}(d(v)-2)-2\cdot 3-f_3\\
&=1+e-f_3\ge1>0.
\end{align*}

\section{Final Remarks}
In Figure~\ref{cluster-two}, we illustrate two matching assignments in a $6$-cluster and a $7$-cluster that prevent us from finding an independent set of order $6$. The examples show that Lemma~\ref{555} and Lemma~\ref{556} cannot be improved.  The examples also show that it would be difficult to improve our results by removing the requirement of forbidding butterflies.  For example, in our proof, $7^+$-vertices have no burden to afford the charges, but if we allow butterflies, even a $8$-vertex may not be able to afford to give charges when it is on two $6^+$-clusters.

\begin{figure}[H]
\includegraphics[scale=0.65]{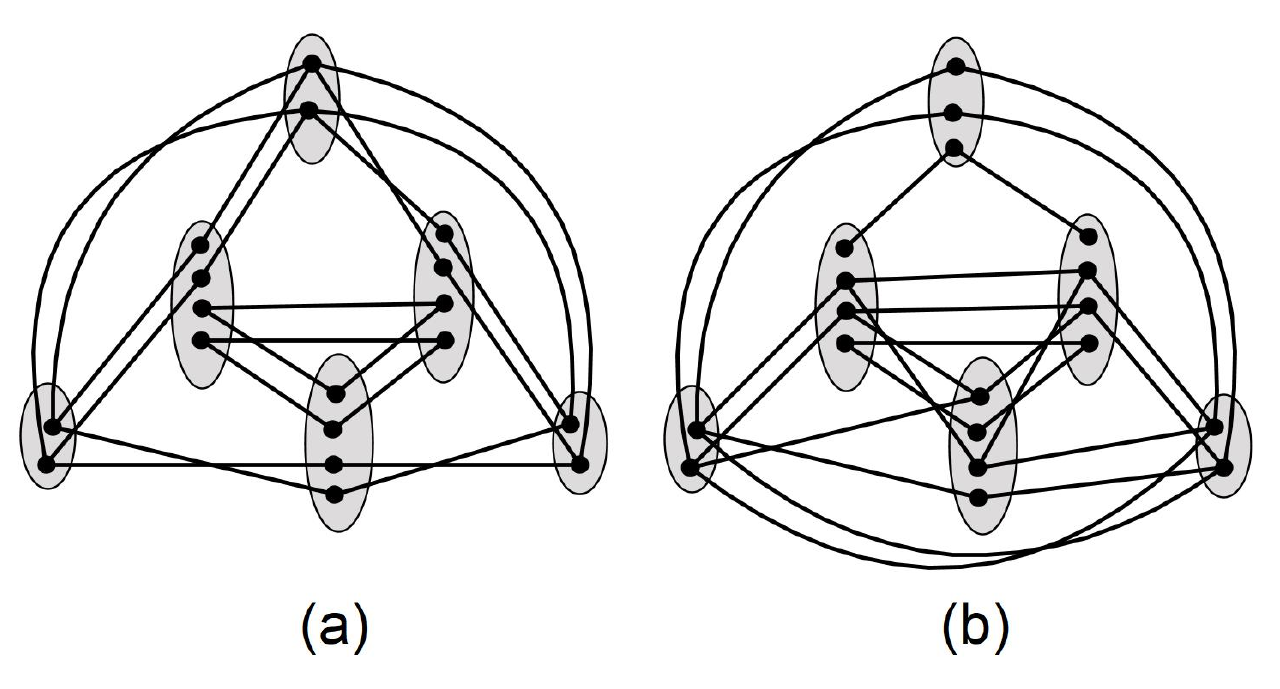}
\caption{}
\label{cluster-two}
\end{figure}

\noindent {\bf Acknowledgement.} The work is done while the second author is studying at the College of William and Mary as a visiting student, supported by the Chinese Scholarship Council. Seog-Jin Kim's work was supported by the National Research Foundation of Korea(NRF) grant funded by the Korea government(MSIT) (NRF-2018R1A2B6003412). Gexin Yu's work was supported in part by the Natural Science Foundation of China (11728102) and the NSA grant H98230-16-1-0316.  The authors would like to thank the referees for their valuable comments.

\end{document}